\documentclass[draft]{article}
\usepackage{color}
\usepackage[utf8]{inputenc}
\usepackage[T1]{fontenc}
\usepackage{lmodern}
\usepackage{xspace}
\usepackage{amsfonts,amsmath,amssymb,amsthm}
\usepackage[small]{titlesec}
\usepackage{comment}
\usepackage[a4paper,body={160mm,235mm},centering]{geometry}
\def\Ac{{\cal A}}

\def\Cc{{\cal C}}

\def\Fc{{\cal F}}

\def\Hc{{\cal H}}

\def\Sc{{\cal S}}

\def \P{\mathbb{P}}

\def \R{\mathbb{R}}

\def\Esp#1{\mathbb{E}\left[#1\right]}

\newcommand{\rset}{\mathbb{R}}

\newcommand{\ep}{\varepsilon}
\newcommand{\ind}{\mathbf{1}}
\newcommand{\fl}{\longrightarrow}
\newcommand{\e}{\mathbb{E}}
\newcommand{\p}{\mathbb{P}}
\newcommand{\lp}{\mathrm{L}}
\newcommand{\m}{\mathcal} 
\newcommand{\dotafter}[1]{#1.}
\titleformat{\section}[hang]
{\normalfont\large\bfseries}{\thesection.}{.5em}{\dotafter}[]
\titleformat{\subsection}[hang]
{\normalfont\bfseries}{\thesubsection.}{.4em}{\dotafter}[]
\titleformat{\paragraph}[hang]{\normalfont\bfseries}{\theparagraph.}{.4em}{\dotafter}[]

\theoremstyle{plain}
\newtheorem{thm}{Theorem}
\newtheorem{lemme}[thm]{Lemma}
\newtheorem{prop}[thm]{Proposition}
\newtheorem{cor}[thm]{Corollary}

\theoremstyle{definition}
\newtheorem*{df}{Definition}

\theoremstyle{remark}
\newtheorem{rem}{Remark}

\definecolor{darkred}{rgb}{0.8,0,0}
\definecolor{darkblue}{rgb}{0,0,0.7}
\definecolor{darkgreen}{rgb}{0,0.4,0}

\title{\bf BSDEs with mean reflection}
\author{Philippe Briand\thanks{Laboratoire de Math\'ematiques CNRS UMR 5127,
Universit\'e de Savoie, Campus Scientifique, 73376 Le Bourget du Lac, France
(philippe.briand@univ-savoie.fr)} \and 
Romuald Elie\thanks{LAMA UMR CNRS 8050, Universit\'e Paris-Est $\&$ Projet MathRisk ENPC-INRIA-UMLV,
 Champs-sur-Marne, France
(romuald.elie@univ-mlv.fr). Research partially supported by the ANR grant LIQUIRISK and the Chair Finance and Sustainable Development.} 
\and Ying Hu\thanks{IRMAR UMR CNRS 6625, Université Rennes 1, Campus de Beaulieu, 35042 Rennes Cedex, France (ying.hu@univ-rennes1.fr), partially supported by Lebesgue center of mathematics ("Investissements d'avenir"
program - ANR-11-LABX-0020-01 and by ANR-15-CE05-0024-02.}}
\begin{document}
\maketitle
\begin{abstract}

 In this paper, we study a new type of BSDE, where the distribution of the $Y$-component of the solution is required to satisfy an additional constraint, written in terms of the expectation of a loss function. 
 This constraint is imposed at any deterministic time $t$ and is typically weaker than the classical pointwise one associated to reflected BSDEs. Focusing on solutions $(Y,Z,K)$ with deterministic $K$, we obtain the well-posedness of such equation, in the presence of a natural Skorokhod type condition. Such condition indeed ensures the minimality of the enhanced solution, under an additional structural condition on the driver. 
%
%
%
%
 Our results extend to the more general framework where the constraint is written in terms of a static risk measure on $Y$. In particular, we provide an application to the super hedging of claims under running risk management constraint. 

\end{abstract}

\section{Introduction}

 Backward Stochastic Differential Equations (BSDEs) have been introduced by Pardoux and Peng \cite{parpen90} and share a strong connection with stochastic control problems. Solving a BSDE typically consists in the obtention of an adapted couple process $(Y,Z)$ with the following dynamics:
\begin{eqnarray*}
Y_t &=& \xi + \int_t^T f(s,Y_s,Z_s) ds - \int_t^T Z_s \cdot dB_s\;, \qquad 0\le t \le T\;.
\end{eqnarray*}

In their seminal paper, Pardoux and Peng  provide the existence of a unique solution $(Y,Z)$ to this equation for a given square integrable terminal condition $\xi$ and a Lipschitz random driver $f$. Since then, many extensions have been derived in several directions. The regularity of the driver can for example be weakened. The underlying dynamics can be fairly more complex, via for example the addition of jumps. These extensions allow in particular to provide representation of solutions to a large class of stochastic control problems, and to tackle several meaningful applications in mathematical finance.\\

 More interestingly, the consideration of additional conditions on the stochastic control problems of interest naturally led to the consideration of constrained BSDEs. In such a case, the solution of a constrained BSDE contains an additional adapted increasing process $K$, such that $(Y,Z,K)$ satisfies
\begin{eqnarray}\label{Constrained_BSDE}
Y_t &=& \xi + \int_t^T f(s,Y_s,Z_s) ds - \int_t^T Z_s \cdot dB_s + K_T - K_t \;, \qquad 0\le t \le T,
\end{eqnarray}
together with a chosen constraint on the solution. The process $K$ interprets as the extra cost on the value process $Y$, due to the additional constraint. In such a framework, this equation admits an infinite number of solutions, as the roles of $Y$ and $K$ are too closely connected. The underlying stochastic control problem of interest typically indicates that one should look for the minimal solution (in terms of $Y$) of such equation. Motivated by optimal stopping or related obstacle problems, El Karoui et al. \cite{EKP} introduced the notion of reflected BSDE, where the constraint is of the form 
\begin{eqnarray*}\label{Pointwise constraint}
Y_t &\ge&  L_t \;, \qquad 0\le t \le T.
\end{eqnarray*}
The obstacle process $L$ is a lower bound on the solution $Y$ and interprets as the reward payoff, if one chooses to stop immediately. It is worth noticing that the minimal solution $(Y,Z,K)$ is fully characterized by the following so-called Skorokhod condition 
\begin{eqnarray*}
\int_0^T (Y_t -L_t)^+ dK_t &=& 0.
\end{eqnarray*}
This condition intuitively indicates that the process $K$ is only allowed to push the value process $Y$ whenever the constraint is binding.\\

 The class of constrained BSDEs has been significantly enlarged in the recent literature. The resolution of zero sum Dynkin game problems led Cvitanic and Karatzas \cite{CK} to the study of doubly reflected BSDE, where the process $Y$ lies in between two processes. 
 Considering super-hedging problems where the admissible portfolios are restricted to belong to a convex set $\textbf{C}$ (e.g. $\textbf{C}=\rset^+$ for no short sell constraints), Buckdahn and Hu \cite{BuHu1,BuHu2} and Cvitanic et al. \cite{CKS} studied the well posedness of BSDE (\ref{Constrained_BSDE}) together with the constraint: $Z_t\in\textbf{C}$, for $t\in[0,T]$. More generally, Peng and Xu \cite{PenXu} considered pointwise constraints of the form $\varphi(t,Y_t,Z_t)\ge 0$, where $\varphi$ is non-decreasing in $y$. The study of optimal switching problems \cite{HamJea,HamZha,HuTan,CEK} led to the consideration of multidimensional  systems of BSDEs with oblique reflections.\\ 
 In contrast to the previously mentioned pointwise constraints on the solution, Bouchard et al. \cite{BER} introduced the notion of BSDE with weak terminal condition. In their framework, the terminal condition is replaced by a constraint on the distribution of the random variable $Y_T$ and would typically rewrite
\begin{eqnarray*}
 \e[\ell(Y_T-\xi)] \ge 0.
\end{eqnarray*}
The term $\ell(X_T-\xi)$ identifies to the quantification of a loss depending on the distance between $Y_T$ and the target $\xi$. This type of BSDE relates in  particular to quantile hedging or related controlled loss control problems.\\

The purpose of this paper is to determine the impact of a dynamic version of such type of constraint, by studying the BSDE (\ref{Constrained_BSDE}) together with a running constraint in expectation of the form
\begin{eqnarray}\label{runningconstraint}
 \e[\ell(t,Y_t)] \ge 0\;, \qquad 0\le t \le T,
\end{eqnarray}
where $(\ell(t,.))_{0\le t\le T}$ is a collection of non decreasing possibly random functions. It is worth noticing that the previous running constraint is only imposed on deterministic times $t\in[0,T]$. In the spirit of the above mentioned  Skorokhod condition for reflected BSDEs, we look towards so-called flat solutions, i.e. satisfying the extra condition
\begin{eqnarray} \label{skorokhodcondition}
\int_0^T \e[\ell(t,Y_t)]  dK_t &=& 0. 
\end{eqnarray}
 Whenever $K$ is allowed to be random, we observe that the construction of a minimal continuous (in $Y$) solution to the system (\ref{Constrained_BSDE})-(\ref{runningconstraint}) is not possible in general. For such reason, we focus on the derivation of solutions $(Y,Z,K)$ with deterministic $K$ component. Whenever $\ell$ is deterministic linear, we provide an explicit construction of the unique flat solution to the system (\ref{Constrained_BSDE})-(\ref{runningconstraint})-(\ref{skorokhodcondition}). For general loss functions, we are also able  to derive the well-posedness of the system (\ref{Constrained_BSDE})-(\ref{runningconstraint})-(\ref{skorokhodcondition}), under mild assumptions satisfied for example whenever $\ell$ is bi-Lipschitz. Besides, restricting to drivers with any dependence in $z$ but deterministic linear dependence in $y$, we verify that the condition (\ref{skorokhodcondition}) ensures the minimality in $Y$ of the enhanced solution, among all considered solutions of the BSDE (\ref{Constrained_BSDE}) with mean reflexion (\ref{runningconstraint}).\\
 
 In terms of applications, it is worth noticing that the constraint (\ref{runningconstraint}) can easily be replaced by a more general version of the form 
\begin{eqnarray}\label{risk_measure_constraint}
 \rho(t,Y_t) \le q_t\;, \qquad 0\le t \le T\;,
\end{eqnarray}
 where $(\rho(t,.))_t$ is a time indexed collection of static risk measures, and $(q_t)_t$ are associated benchmark levels. This framework is in fact the main motivation of this paper, but we chose to present our main argumentation within the constraint (\ref{runningconstraint}) for sake of clarity and simplicity. We detail in particular in the last section of the paper an application to the super-replication of claims, restricting to investment portfolio $Y$ satisfying risk management constraint of the form \eqref{risk_measure_constraint}.\\

The paper is organized as follows: Section \ref{sec:setup} presents the problem of interest, clarifies the assumptions and discusses the main results of the paper.  In Section \ref{sec:llin}, we build the unique solution to the system (\ref{Constrained_BSDE})-(\ref{runningconstraint})-(\ref{skorokhodcondition}) whenever $\ell$ is linear and deterministic. The general case is treated in Section \ref{sec:general_case}, where we derive the well-posedness of the system (\ref{Constrained_BSDE})-(\ref{runningconstraint})-(\ref{skorokhodcondition}). The minimality  of the enhanced solution is discussed in Section \ref{sec:minimal}, whereas the mathematical finance application is given in Section \ref{sec:application}.\\

\paragraph{Notations} 

Throughout this paper, we are given a finite horizon $T$ and a complete probability space $(\Omega,\Fc,\P)$ endowed with a $d$-dimensional standard Brownian motion  $B=(B_t)_{0\geq t\leq T}$. We will work with the usual augmented filtration of $B$, $\{ \m F_t\}_{0\leq t\leq T}$. Any element $x \in \R^d$ will be identified to a column vector with $i$-th component $x^i$ and Euclidian norm $|x|$. $\Cc_T$ denotes the set $C([0,T],\R)$ of continuous functions from $[0,T]$ to $\R$. For a given set of parameters $\alpha$, $C(\alpha)$ will denote a constant only depending on these parameters, and which may change from line to line.
Finally, we classically denote by:\\[-5mm]
  \begin{itemize}
\setlength\itemsep{0pt}
\setlength\leftmargin{\parindent}
\setlength\itemindent{-\parindent}
\item $L^2(\Fc_t)$ the set of real valued $\Fc_t$-measurable square integrable random variables, for any $t\in[0,T]$.
\item $\Sc^2$ the set of real valued $\Fc$-adapted  continuous 
processes $Y$ on $[0,T]$ such that\\
    $
    \|Y\|_{\Sc^2} := \Esp{\sup_{0\le r \le T} |Y_r|^2}^{\frac{1}{2}} <\infty
    $;
 \item $\Hc^2$ the set of predictable $\R^d$-valued processes $Z$ s.t.
    $
    \!\|Z\|_{\Hc^2}\! := \! \Esp{\int_0^T |Z_r|^2 dr}^\frac{1}{2}\!\!\!\!~<~\!\!\infty
    $;
\item ${\Ac^2}$ is the  closed subset of $\Sc^2$ consisting of
nondecreasing  processes $K$ $=$ $(K_t)_{0\leq t\leq T}$ with $K_0$
$=$ $0$;
\item ${\Ac^2_D}$ the subset of deterministic elements of ${\Ac^2}$.
%
\end{itemize}

\section{Problem set up}\label{sec:setup}

\subsection{Presentation of BSDEs with mean reflexion}

The main purpose of this paper is to construct solutions $(Y,Z,K)
$ to the following BSDE 
\begin{gather}
	Y_t  =\xi+\int_t^T f(s,Y_s,Z_s)\, ds - \int_t^T Z_s\cdot dB_s + K_T-K_t,\quad 0\leq t\leq T, \label{eq:main_dyn}\\
	\e[\ell(t,Y_t)] \geq 0, \quad 0\leq t \leq T, \label{eq:main_constraint}
\end{gather}
	where the second equation is a running constraint in expectation on the component $Y$ of the solution. In opposition to classical reflected BSDE where \eqref{eq:main_constraint} would typically be a pointwise constraint, the constraint considered here concerns the distribution of the $Y$-component. We 
	pin this new type of constrained equations as \textit{BSDEs with mean reflexion}. 

	The non-decreasing function $\ell$ interprets as a loss function and typical examples of interest  are
	\begin{itemize}
	 \item $\ell(t,x)=x-u_t$ where $u$ is a deterministic continuous benchmark, that the process $Y$ is required to beat in expectation;
	 \item $\ell(t,x) = \ind_{x \geq u_t} - v_t$ (or any smoothed equivalent), so that the process $Y$ is now required to beat deterministic continuous benchmark $u$ with a probability greater than $v_t$, for any time $t$;
	 \item $\ell(t,x) = U(x,\xi_t) -u_t$ where $U$ is a concave utility function, $(\xi_t)_t$ is a running random benchmark of interest and $(u_t)_t$ a given deterministic confidence level.
	 \end{itemize}

	Whenever $\ell$ is a strictly increasing function, the corresponding classical reflected BSDE  is characterized by the dynamics \eqref{eq:main_dyn} together with the stronger pointwise constraint 
\begin{eqnarray*}
	\ell(t,Y_t) \geq 0, \quad 0\leq t \leq T.
\end{eqnarray*}
	In such a case, the $Y$-component of the solution to the BSDE is reflected on the boundary process $([\ell(t,.)]^{-1}(0))$. Observe that our constrained BSDE of interest weakens the condition imposed on $Y$, by only constraining its distribution.

\begin{rem}
	Observe that the condition \eqref{eq:main_constraint} is only written on the deterministic dates of $[0,T]$, and not on all the possible stopping times smaller than $T$. In our framework, considering a constraint on all stopping times would strongly strengthen the constraint of interest. On the contrary, both type of pointwise conditions are by construction equivalent for classical reflected BSDEs.
\end{rem}

\subsection{Assumptions on the coefficients}

 The parameters of the BSDE with mean reflection are the terminal condition $\xi$, the driver $f$ as well as the loss function $\ell$. These parameters are supposed to satisfy the following standard running assumptions:
 \begin{itemize}
 \item[($H_f$)] The driver $f:\Omega\times[0,T]\times\rset\times\rset^d \fl \rset$ is a measurable map with respect to $\m P\times \m B(\rset)\times\m B\left(\rset^d\right)$ and $\m B(\rset)$, $\m P$ being the sigma algebra of progressive sets of $\Omega\times[0,T]$, and there exists $\lambda\geq 0$ such that, $\p$-a.s., for all $t\in[0,T]$,
	\begin{equation*}
		\forall y,p,z,q\qquad \left| f(t,y,z) - f(t,p,q)\right| \leq \lambda \left( |y-p| + |z-q|\right),
	\end{equation*}
and
$$ \e\left[ \int_0^T | f(t,0,0) | ^2\, dt\right] <+\infty.$$
\item[($H_\xi$)] The terminal condition $\xi$ is a square-integrable ${\cal F}_T$-measurable random variable such that
$$\e[\ell(T,\xi)] \geq 0.$$
\item[($H_\ell$)] The loss function $\ell : \Omega\times[0,T]\times\rset \fl \rset$ is a measurable map with respect to $\m F_T\times\m B([0,T])\times \m B(\rset)$ and there exists $C\geq 0$ such that, $\p$-a.s.,
\begin{enumerate}
	\item $(t,y)\longmapsto \ell(t,y)$ is continuous,
	\item $\forall t\in[0,T]$, $y\longmapsto \ell(t,y)$ is strictly increasing,
	\item $\forall t\in[0,T]$, $\e\left[\ell(t,\infty)\right]  >0$,
	\item $\forall t\in[0,T]$, $\forall y\in\rset$, $|\ell(t,y)| \leq C(1+|y|)$.
\end{enumerate} 
\end{itemize}

 \begin{rem}
	We chose to work in this paper under that seminal Lipschitz and square integrability assumptions on the driver and terminal condition. 
	 We restrict here to this simple framework, in order to decrease the amount of technical details and emphasize the novelty induced of the additional constraint \eqref{eq:main_constraint}.
 \end{rem}
 
\begin{rem}
	Observe that Condition ($H_\xi$) ensures  that the constraint is automatically satisfied at maturity. This condition implies that no a priori facelift procedure is required on the terminal payoff $\xi$. 
\end{rem}

\subsection{Definition of solution, main results and discussion}

We now turn to the definition of a solution to the BSDE \eqref{eq:main_dyn} with mean reflexion \eqref{eq:main_constraint} of interest.
 
 	\begin{df}
		A square integrable solution to the BSDE~\eqref{eq:main_dyn} with mean reflection~\eqref{eq:main_constraint}  is a triple of processes $(Y,Z,K)\in\Sc^2\times\Hc^2\times\Ac^2$ satisfying \eqref{eq:main_dyn} together with the constraint \eqref{eq:main_constraint}. 
		A solution is said to be \textit{flat} if moreover $K$ increases only when expected, i.e. when we have 
	\begin{equation}\label{main_flat}
		\int_0^T \e[\ell(t,Y_t)] \, dK_t = 0.
	\end{equation}
			By a \textit{deterministic solution}, we mean a solution for which the process $K$ is deterministic., i.e. $K\in\Ac^2_D$.
%
%
	\end{df} 
	
 

As detailed in Remark  \ref{rem_uniqueness} below, we observe that allowing $K$ to be random leads to the existence of multiple flat solutions. We even verify that it may induce the non-existence of minimal solution  for the BSDE \eqref{eq:main_dyn} with mean reflection \eqref{eq:main_constraint}, see the example provided at the end of Section \ref{sec:minimal}. This is why we chose here to restrict to the consideration of so-called \textit{deterministic solutions}, i.e. solutions $(Y,Z,K)$ with deterministic compensator $K$.

 In particular, focusing on deterministic solutions, we verify that the flatness condition \eqref{main_flat} can directly imply the minimality property of the solution beyond all the deterministic ones. This is in particular the case for drivers with deterministic linear dependence in $y$, see Condition \eqref{struct_f}.\\

The main result of this paper is the existence and uniqueness of deterministic flat solution to the BSDE \eqref{eq:main_dyn} with mean reflection \eqref{eq:main_constraint}.
This is first achieved  for the particular case of linear loss function $\ell$, see Proposition \ref{prop:existence1_linear} and Theorem \ref{thm existence llinear} in Section \ref{sec:llin}. The line of proof follows a constructive approach when the driver does not depend on $Y$ and $Z$, together with a contraction property in order to tackle any Lipschitz driver function. An alternative approach via penalization is also provided in Section \ref{subsec:penalize}. 
When the driver is not linear, the well posedness of the system \eqref{eq:main_dyn}-\eqref{eq:main_constraint}-\eqref{main_flat} is also established, under an additional assumption on the loss function, denoted ($H_L$) below, see Proposition \ref{prop:existence_uniqueness_2} and Theorem \ref{thm: main}  in Section \ref{sec:general_case}.\\

%
%
%
%
%
%
%
%
%
%

 	In a similar fashion, we explain in Section~\ref{sec:application} below how the constraint in expectation  \eqref{eq:main_constraint} can be replaced by a constraint of the form $\rho(\cdot,Y_\cdot)\le q_\cdot$, where $(\rho(t,\cdot))_t$ is a collection of static risk measures computed at time $0$, and $q$ is a collection  of time-indexed benchmarks. In particular, solving this equation allows for example to represent the super-hedging price of a claim $\xi$, whenever any admissible portfolios require to satisfy at any date $t$ a running risk management constraint written in terms of risk measures. \\
	
	\begin{rem}
		Since the constraint \eqref{eq:main_constraint} concerns the distribution of the solution to the BSDE, it is tempting to understand the possible connection between such type of BSDE and corresponding constrained McKean Vlasov BSDEs. This topic seems promising in particular for the mean field game literature and is left for further research.
	\end{rem}

\subsection{A priori estimate}

 Let us conclude this section by providing a usefull a priori estimate on any solution to the BSDE \eqref{eq:main_dyn}-\eqref{eq:main_constraint} of interest. 
%

\begin{lemme}\label{en:S2}
	Let $(Y,Z,K)$ be a square integrable solution to the BSDE \eqref{eq:main_dyn} with mean reflection \eqref{eq:main_constraint}. Then $Y$ satisfies the following 
	\begin{equation*}
		\e\left[ \sup_{0\leq t \leq T} |Y_t|^2\right] \leq C(\lambda, T) \, \e\left[ |Y_0|^2 + K_T^2 + \int_0^T |f(s,0,0)|^2 ds + \int_0^T |Z_s|^2 ds\right].
	\end{equation*}
\end{lemme}

\begin{proof}
	By construction, we have, 
	\begin{equation*}
		Y_t = Y_0 - \int_0^t f(s,Y_s,Z_s)\, ds + \int_0^t Z_s\cdot dB_s - K_t, \qquad 0\le t \le T\;.
	\end{equation*}
	 Because $K$ is non decreasing, Assumption $(H_f)$ leads to 
	\begin{equation*}
		|Y_t| \leq |Y_0| + K_T + \int_0^T |f(s,0,0)| ds + \lambda \int_0^T |Z_s| ds + \sup_{0\leq t \leq T} \left | \int_0^t Z_s\cdot dB_s \right| + \lambda \int_0^t |Y_s|\, ds \,,
	\end{equation*}
	for $t\in[0,T]$. Since $Y$ has continuous paths, Gronwall's lemma gives
	\begin{equation*}
		\sup_{0\leq t \leq T} |Y_t| \leq e^{\lambda T} \left(|Y_0| + K_T + \int_0^T |f(s,0,0)| ds + \lambda \int_0^T |Z_s| ds + \sup_{0\leq t \leq T} \left | \int_0^t Z_s\cdot dB_s \right|\right) \,,
	\end{equation*}
	and the result follows from the Burkholder-Davis-Gundy inequality.
\end{proof}

\begin{rem}
	We deduce from this lemma that, when the generator has linear growth, the process $Y$ belongs to $\m S^2$ as soon as $Z$ and $K$ are square integrable.
\end{rem}

\section{ The particular case of linear mean reflection}
\label{sec:llin}

In this section, we consider the simpler particular case where the mean reflection is linear. Namely, $\ell:(t,y)\mapsto y-u_t$ so that  the condition \eqref{eq:main_constraint} is replaced by 
\begin{eqnarray}
	 \e[Y_t] \geq u_t, \qquad 0\leq t \leq T, \label{eq:main_constraint_linear}
\end{eqnarray}
where $u$ is a deterministic continuous map from $[0,T]$ to $\rset$. Hereby, we impose a running deterministic lower bound $u$ on the expected value of the $Y$-component of the solution. Besides, we recall that Assumption $H_\xi$ ensures that this constraint is already satisfied at maturity so that we have
\begin{eqnarray}
	 \e[\xi] \geq u_T \,. \qquad  \label{eq:terminal_constraint_linear}
\end{eqnarray}

In this linear framework, we are able in Proposition \ref{prop:existence1_linear}  to construct an explicit deterministic flat solution $(Y,Z,K)$ to a BSDE with linear mean reflexion \eqref{eq:main_constraint_linear}, when the driver does not depend on $Y$ nor $Z$. Building modifications on this deterministic flat solution, we exhibit an infinite number of non deterministic flat solutions to the same BSDE. This feature is our main motivation in order to focus solely on deterministic flat solutions in order to ensure the well posedness of BSDEs  with mean reflection. Indeed, Proposition \ref{prop:uniqueness_linear} indicates that uniqueness holds within the class of deterministic flat solutions to \eqref{eq:main_dyn}-\eqref{eq:main_constraint_linear}.\\ 

Hereafter, we first derive an a priori estimate on the solution, and then tackle respectively the uniqueness and existence issues. In order to handle general drivers, the enhanced demonstration relies on a contraction argument, but an alternative approach via penalization is also provided in Section \ref{subsec_penalization}.




\subsection{A priori estimate}\label{subsec_estimate}

 The main mathematical advantage of considering a linear loss function $\ell$ is that it allows to use some of the computational tricks associated to classical reflected BSDEs, in particular when the compensator $K$ is moreover deterministic.  As detailed in the proof below, this enables us to derive the following a-priori estimate on the solution to the BSDE with linear mean reflexion.

\begin{lemme}\label{en:est}
	Let $(Y,Z,K)$ be a deterministic square integrable flat solution to the BSDE \eqref{eq:main_dyn} with linear mean reflexion \eqref{eq:main_constraint_linear}. Then
	\begin{equation*}
		\e\left[\sup_{0\leq t \leq T} |Y_t|^2 + \int_0^T |Z_s|^2 ds\right] + K_T^2 \leq C(\lambda, T) \left(\e\left[|\xi|^2 + \int_0^T |f(s,0,0)|^2 ds\right] + \| u \|_\infty^2\right).
	\end{equation*}
\end{lemme}

\begin{proof}
	Let us recall that the Lipschitz property of $f$ implies 
	\begin{equation*}
		2 y\cdot f(t,y,z) \leq |f(t,0,0)|^2 + \frac{1}{2} |z|^2 + \left(1+2\lambda + 2\lambda^2\right) |y|^2, \qquad \forall (y,z)\in\R\times\R^d\;.
	\end{equation*}
	Setting $\beta := 1 + 2\lambda + 2 \lambda^2$, Itô's formula together with the previous inequality provides
	\begin{multline*}
		e^{\beta t} |Y_t|^2 + 
		\frac{1}{2} \int_t^T e^{\beta s} |Z_s|^2 ds   \leq e^{\beta T} |\xi|^2 + \int_0^T e^{\beta s} |f(s,0,0)|^2 ds + 2 \int_t^T e^{\beta s} Y_s dK_s -2 \int_t^T e^{\beta s} Y_s Z_s\cdot dB_s,
	\end{multline*}
	for all $t\in[0,T]$. Since $K$ is deterministic and $\ell$ is linear, we compute
	\begin{align*}
		2 \e \left[\int_t^T e^{\beta s} Y_s dK_s\right] = 2  \int_t^T e^{\beta s} \e\left[Y_s\right] dK_s & = 2  \int_t^T e^{\beta s} \left(\e\left[Y_s\right] -u_s\right)dK_s + 2  \int_t^T e^{\beta s} u_s dK_s \;.
	\end{align*}
	Besides the solution is flat so that condition \eqref{main_flat} directly implies
	\begin{align*}
		2 \e \left[\int_t^T e^{\beta s} Y_s dK_s\right] 
		  & = 2  \int_t^T e^{\beta s} u_s dK_s \leq 2 e^{\beta T} \| u \|_\infty K_T.
	\end{align*}
	We deduce that
	\begin{multline*}
		\sup_{0\leq t\leq T} \e \left[e^{\beta t} |Y_t|^2\right] + \e\left[ \int_0^T e^{\beta s} |Z_s|^2 ds \right] 
		\leq 3 \left( \e\left[e^{\beta T} |\xi|^2 + \int_0^T e^{\beta s} |f(s,0,0)|^2 ds\right] + 2 e^{\beta T} \| u \|_\infty K_T \right),
	\end{multline*}
	from which, we get, for any $\ep>0$,
	\begin{equation}\label{eq:p1}
		\sup_{0\leq t\leq T} \e \left[ |Y_t|^2\right] + \e\left[ \int_0^T  |Z_s|^2 ds \right] \leq C(\lambda,T,\ep) \left(\e\left[|\xi|^2 + \int_0^T |f(s,0,0)|^2 ds\right] + \| u \|_\infty^2\right) + \ep\, K_T^2.
	\end{equation}	
	On the other hand, since $K$ is deterministic, we have
	\begin{equation*}
		K_T = \e\left[K_T\right] = Y_0 - \e\left[\xi\right] - \e\left[\int_0^T f(s,Y_s,Z_s) ds \right],
	\end{equation*}
	from which we deduce the inequality
	\begin{equation}\label{eq:p2}
		K_T^2 \leq C(\lambda,T) \left(\e\left[ \int_0^T |f(s,0,0)|^2 ds \right] + \sup_{0\leq t\leq T} \e\left[|Y_t|^2\right] + \e \left[\int_0^T |Z_s|^2 ds \right] \right).
	\end{equation}	
	Combining this estimate  with  \eqref{eq:p1} 
	and $\ep$ small enough, we get
	\begin{equation*}
		\sup_{0\leq t\leq T} \e \left[ |Y_t|^2\right] + \e\left[ \int_0^T  |Z_s|^2 ds \right] + |K_T|^2\leq C(\lambda,T) \left(\e\left[|\xi|^2 + \int_0^T |f(s,0,0)|^2 ds\right] + \| u \|_\infty^2\right)
	\end{equation*}
	and the result follows from Lemma~\ref{en:S2}.
\end{proof}

\subsection{ Uniqueness of the deterministic flat solution} \label{subsec_unique_linear} 

The uniqueness of flat deterministic solution for a BSDE with linear mean reflection follows mainly from a similar argumentation as the one used for classical reflected BSDE. This is detailed in the next Proposition. 

\begin{prop}\label{prop:uniqueness_linear}
	The BSDE \eqref{eq:main_dyn} with linear mean reflexion \eqref{eq:main_constraint_linear}  has at most one square integrable deterministic flat solution.
\end{prop}

\begin{proof}
	Let us consider two such solutions $(Y^1, Z^1, K^1)$ and $(Y^2, Z^2, K^2)$ and denote
	$$\delta Y:=Y^1-Y^2,\quad \delta Z:=Z^1-Z^2 \;\;\mbox{ and } \;\; \delta K:=K^1-K^2.$$	
	 Setting $a:=2\lambda + 2 \lambda^2$ and arguing as in Lemma \ref{en:est}, Itô's formula gives easily
	\begin{equation*}
		e^{a t}|\delta Y_t |^2 + \frac{1}{2 }\int_t^T e^{as} |\delta Z_s|^2 \, ds   \leq  -  2 \int_t^T e^{as} \delta Y_s\,\delta Z_s\cdot dB_s +  2 \int_t^T e^{as} \delta Y_s\, d\delta K_s \,,
	\end{equation*}
	for $t\in[0,T]$.	Let us observe that since both solutions are flat and deterministic and $\ell$ is linear, we nicely have
	\begin{align*}
		\e\left[\int_t^T e^{as} \delta Y_t \, d \delta K_s\right] & = \int_t^T e^{as} \left[\left(\e[Y^1_s]-u_s\right) -  \left(\e\left[Y^2_s\right]-u_s\right)\right] d K^1_s  \\
		& \quad - \int_t^T e^{as} \left[\left(\e[Y^1_s]-u_s\right) - \left(\e\left[Y^2_s\right]-u_s\right)\right] d K^2_s \\
		& = - \int_t^T e^{as} \left(\e\left[Y^2_s\right]-u_s\right) d K^1_s - \int_t^T e^{as}  \left(\e\left[Y^1_s\right]-u_s\right) d K^2_s \leq 0,
	\end{align*}
	for any $t\in[0,T]$. Thus the result follows by taking expectations in the previous inequality.
\end{proof}

As detailed in Remark  \ref{rem_uniqueness} below, considering deterministic $K$ processes is a key for the obtention of a unique solution to the BSDE of interest. We now turn to the existence property.

\subsection{ Existence of a deterministic flat solution}  \label{subsec_existence_linear}

We first focus on the particular case where the driver $f$ does not depend on $Y$ nor $Z$. In this simple case, we are able to construct explicitly the unique solution to a BSDE with linear mean reflection.

\begin{prop}\label{prop:existence1_linear}
 Let $C$ be a square integrable progressively measurable stochastic process or more generally in the space $\lp^2\left(\Omega;\lp^1(0,T)\right)$. The BSDE with linear mean reflection
\begin{equation}\label{eq:linov}
	Y_t  =\xi+\int_t^T C_s\, ds - \int_t^T Z_s\cdot dB_s + K_T-K_t,\qquad \e[Y_t] \geq u_t, \qquad 0\leq t\leq T,
\end{equation}
 has a unique square integrable deterministic flat solution.
\end{prop}

\begin{proof}
	Let us set $\displaystyle x_t = \e\left[\xi + \int_t^T C_s\, ds\right]$. By Skorokhod's lemma, there exists a unique pair of deterministic functions $(y,K):[0,T]\rightarrow\rset$ such that $K$ is non decreasing and $K_0=0$ and we have
	\begin{equation}\label{temp1}
		y_t = x_t + K_T-K_t,\qquad y_t \geq u_t,\qquad \int_0^T (y_t-u_t) \, dK_t=0.
	\end{equation}
	By construction, observe that $K$ is continuous and  $K_t = \sup_{0\leq s \leq T} \left(x_s-u_s\right)_- - \sup_{t\leq s\leq T}\left(x_s-u_s\right)_-$.  Now, $K$ being given, we know that the BSDE
	\begin{equation*}
		Y_t  =\xi+\int_t^T C_s\, ds - \int_t^T Z_s\cdot dB_s + K_T-K_t,\quad 0\leq t\leq T,
	\end{equation*}
	has a unique square integrable solution $(Y,Z)$. Moreover, we have by construction $y_t = \e[Y_t]$. It follows from \eqref{temp1} that $(Y,Z,K)$ is a deterministic flat solution of the BSDE~\eqref{eq:linov}. The uniqueness follows from Proposition \ref{prop:uniqueness_linear}.
\end{proof}

\begin{rem}\label{rem_uniqueness}
	Let us observe that the BSDE with mean reflexion \eqref{eq:linov} has infinite many  flat solutions with random $K$. Let us start with $(Y^0,Z^0,K^0)$ the deterministic flat solution to \eqref{eq:linov} constructed above in the proof of Proposition \ref{prop:existence1_linear}. For any real $\alpha$, let us set $M^\alpha := (e^{\alpha B_t - \alpha^2 t/2})_t$ and define 
	\begin{equation*}
		K^\alpha_t := \int_0^t  M^\alpha_s\, dK^0_s, \qquad 0\le t \le T\;.
	\end{equation*}
	Being given $K^\alpha$, let $(Y^\alpha,Z^\alpha)$ be the solution to the BSDE
	\begin{equation*}
		Y^\alpha_t  =\xi+\int_t^T C_s\, ds - \int_t^T Z^{\alpha}_s\cdot dB_s + K^{\alpha}_T-K^{\alpha}_t,\quad 0\leq t\leq T.
	\end{equation*}
	For all $0\le t \le T$, since $\e[M^\alpha_t]=1$ and $K^0$ is deterministic, we have $\e\left[K^{\alpha}_t\right] = K^0_t$ so that $\e\left[Y^\alpha_t\right] = \e\left[Y^0_t\right] \geq u_t$. Moreover, since $\e\left[Y^0_t\right]-u_t = 0$ $dK$-a.e., we compute
	\begin{equation*}
		\int_0^T \left(\e\left[Y^{\alpha}_t\right] - u_t\right) dK^{\alpha}_t = \int_0^T \left(\e\left[Y^0_t\right] - u_t\right)  M^\alpha_t \, dK^0_t = 0.
	\end{equation*}
	Hence, for any real $\alpha$, $\left(Y^{\alpha},Z^{\alpha},K^{\alpha}\right)$ is also a flat solution to \eqref{eq:linov}.
\end{rem}

We are now in position to turn to the general driver case and we will derive the well-posedness of the BSDE of interest via the classical use of a well chosen contraction property. 

\begin{thm}\label{thm existence llinear}
	The BSDE \eqref{eq:main_dyn} with linear mean reflexion \eqref{eq:main_constraint_linear} has a unique deterministic square integrable flat solution.
\end{thm}

\begin{proof}
	For given processes $U\in\Sc^2$ and $V\in\Hc^2$, let $(Y,Z,K)$ be the deterministic square integrable flat solution to the BSDE
	\begin{equation*}
		Y_t  =\xi+\int_t^T f(s,U_s,V_s)\, ds - \int_t^T Z_s\cdot dB_s + K_T-K_t, \qquad \e\left[Y_t\right] \geq u_t,\qquad 0\le t \le T\;,
	\end{equation*}
	as provided by Proposition \ref{prop:existence1_linear}. Let us show that the mapping $\Phi:(U,V) \longmapsto (Y,Z)$, from $\m S^2\times \m H^2$ into itself, has a unique fixed point. 
	
	For this purpose, let us denote $(Y^1,Z^1,K^1)$ and $(Y^2,Z^2,K^2)$   the two deterministic square integrable flat solutions to the above  BSDE with given $(U^1,V^1)$ and $(U^2, V^2)$ respectively. Set
	$$\delta Y:=Y^1-Y^2,\quad \delta Z:=Z^1-Z^2,\quad \delta K:=K^1-K^2, \quad \delta U:=U^1-U^2, \quad 
	\delta V:=V^1-V^2.$$
	For $a=4\lambda^2 +1$, Itô's formula leads to
	\begin{multline*}
		|\delta Y_0 |^2 + \int_0^T e^{as} \left(|\delta Y_s|^2 + |\delta Z_s|^2\right) \, ds  \\
		 \leq \frac{1}{2} \int_0^T e^{as} \left( |\delta U_s|^2 + |\delta V_s|^2\right) ds -  2\int_0^T e^{as} \delta Y_s\,\delta Z_s\cdot dB_s +  2 \int_0^T e^{as} \delta Y_s\, d\delta K_s.
	\end{multline*}
	As observed in the proof of Proposition \ref{prop:existence1_linear}, we compute
	\begin{align*}
		\e\left[\int_0^T e^{as} \delta Y_s \, d \delta K_s\right] 
		& = - \int_0^T e^{as} \left(\e\left[Y^2_s\right]-u_s\right) d K^1_s - \int_0^T e^{as}  \left(\e\left[Y^1_s\right]-u_s\right) d K^2_s \leq 0.
	\end{align*}	
	It follows directly that
	\begin{equation*}
		\e\left[\int_0^T e^{as} \left(|\delta Y_s|^2 + |\delta Z_s|^2\right) \, ds\right] \leq \frac{1}{2} \e\left[\int_0^T e^{as} \left( |\delta U_s|^2 + |\delta V_s|^2\right) ds\right].
	\end{equation*}
	Since we have
	\begin{gather*}
		\delta Y_t = \e\left[ \int_t^T \left(f(s,U^1_s,V^1_s)-f(s,U^2_s,V^2_s)\right)ds \: \Big|  \: \m F_t\right] + (K^1_T-K^1_t) - (K^2_T-K^2_t), \\
		K^i_T-K^i_t = \sup_{t\leq s\leq T} \left(\e\left[\xi+\int_s^T f\left(r,U^i_r,V^i_r\right)dr\right]-u_s\right)_- \;,
	\end{gather*}
	we get immediately 
	\begin{equation*}
		\e\left[\sup_{0\leq t\leq T} |\delta Y_t|^2\right] \leq C\, \e\left[\int_0^T \left( |\delta U_s|^2 + |\delta V_s|^2\right) ds\right].
	\end{equation*}
	As a byproduct, $\Phi$ is continuous from $\m S^2\times\m H^2$ into itself.
	
	Moreover, starting from $(Y^0,Z^0) = (0,0)$ and setting for $n\geq 1$, $(Y^n,Z^n) = \Phi\left(Y^{n-1},Z^{n-1}\right)$, we deduce easily from the previous estimates that
	\begin{equation*}
		\e\left[\sup_{0\leq t\leq T} \left|Y^{n+1}_t -Y^n_t\right|^2 + \int_0^T \left|Z^{n+1}_t -Z^n_t\right|^2 dt\right] \leq C\, 2^{-n},
	\end{equation*}
	and finally that the sequence $\left\{(Y^n,Z^n)\right\}_{n\geq 0}$ converges in $\m S^2\times\m H^2$ to the unique fixed point of $\Phi$.
\end{proof}

\subsection{ Alternative approach via penalization} \label{subsec_penalization}\label{subsec:penalize}

 In order to handle classical reflected BSDE, a very helpful feature is the characterization of the solution as a limit of corresponding penalized classical BSDEs. The idea simply relies on the addition of a strong penalization on the driver of a classical BSDE, which is only active whenever the constraint is not satisfied. As the penalization strength increases, the $Y$ component of the penalized solution also increases and converges at the limit to the minimal solution of the reflected BSDE. In our framework, the constraint only integrates the distribution of $Y$, and not the pointwise value of the process $Y$. For this reason, no comparison argument can ensure that a sequence of penalized BSDEs will be increasing and the classical line of proof falls down. Nevertheless, whenever the benchmark function $u$ is constant, we are able to identify the unique deterministic flat solution of a BSDE with linear mean reflexion as the limit of corresponding penalized BSDEs of McKean-Vlasov type.
 This is the purpose of the next Proposition.

 \begin{prop} Suppose that the benchmark $(u_t)_t$ is constant and also denoted $u$. For any positive integer $n$, let us consider $(Y^n,Z^n)$ solution to the BSDE of McKean-Vlasov type
	\begin{align*}
		Y^n_t & = \xi + \int_t^T f\left(s,Y^n_s,Z^n_s\right) ds + \int_t^T n \left(u-\e\left[Y^n_s\right]\right)_+ ds - \int_t^T Z^n_s\cdot dB_s \,. \quad 0\le t \le T\;,\\
	\end{align*}
	and denote $K^n:= \int_0^. n \left(u-\e\left[Y^n_s\right]\right)_+ ds$.  As $n$ goes to infinity, $(Y^n,Z^n,K^n)$ converges to the unique flat deterministic solution of the BSDE \eqref{eq:main_dyn} with linear mean reflexion \eqref{eq:main_constraint_linear}. 
 \end{prop}

 \begin{proof} Observe first the the solution $(Y^n,Z^n)$ is well and uniquely defined, according to the results of \cite{Buckdahn} up to slight modifications discussed for example in \cite{ChaGar}.  
 
 \paragraph{Step 1. Uniform a priori estimate on the sequence $(Y^n,Z^n,K^n)_n$}
 Since $K^n$ is deterministic, we have
	\begin{align*}
		2 \e \left[\int_t^T e^{as} Y^n_s dK^n_s\right] = 2  \int_t^T e^{as} \e\left[Y^n_s\right] dK^n_s  & = 2  \int_t^T e^{as} \left(\e\left[Y^n_s\right] -u\right)dK^n_s + 2  \int_t^T e^{as} u dK^n_s \\
		& = - 2n \int_t^T e^{as} \left(u-\e\left[Y^n_s\right]\right)_+^2 ds + 2  \int_t^T e^{as} u dK^n_s \\
		& \leq 2  u \int_t^T e^{as}  dK^n_s\;,
		\end{align*}
		for any constant $a$ and $t\in[0,T]$. 	Thus, arguing as in the proof of Lemma~\ref{en:est}, we get the following estimate on the solution $(Y^n,Z^n)$
	\begin{equation*}
		\sup_{n\geq 1 }\left(\e\left[\sup_{0\leq t \leq T} |Y^n_t|^2 + \int_0^T |Z^n_s|^2 ds\right] + \left|K_T^n\right|^2\right) \leq C(\lambda, T) \left(\e\left[|\xi|^2 + \int_0^T |f(s,0,0)|^2 ds\right] + u ^2\right).
	\end{equation*}
	
 \paragraph{Step 2. Convergence of the sequence $(Y^n,Z^n,K^n)_n$}
 Since the constraint is satisfied at maturity, observe also that $(\left(u-\e\left[Y^n_0\right]\right)_+)^2$ rewrites
\begin{eqnarray*}
|\left(u-\e\left[Y^n_0\right]\right)_+|^2+2n\int_0^T |\left(u-\e\left[Y^n_s\right]\right)_+|^2ds
&=&-2\int_0^T \mathbb E[f(s,Y_s^n,Z_s^n)] \left(u-\e\left[Y^n_s\right]\right)_+ ds\\
&\le& n\int_0^T |\left(u-\e\left[Y^n_s\right]\right)_+|^2ds+\frac{C(\lambda,T)}{n},
\end{eqnarray*}
according to the previous estimate. Hence we deduce for later use that 
 \begin{equation}\label{eq:fonda}
  n^2 \int_0^T |\left(u-\e\left[Y^n_s\right]\right)_+|^2ds \le C(\lambda,T).
  \end{equation}

We now look towards a contracting property of the sequence $(Y^n,Z^n)$ and denote  $\delta X := X^{n+1}-X^n$ for $X=Y,Z$ or $K$.
 Setting $a := \frac{1}{2} + 2\lambda + 2 \lambda^2$, a standard computation based on Itô's formula  provides
	\begin{equation*}
		e^{a t}|\delta Y_t |^2 + \frac{1}{2}\int_t^T e^{as} \left(|\delta Y_s|^2 + |\delta Z_s|^2\right) \, ds  
		 \leq 2 \int_t^T e^{as} \delta Y_s\, d\delta K_s -  2\int_t^T e^{as} \delta Y_s\,\delta Z_s\cdot dB_s, \quad 0\le t \le T,
	\end{equation*}
	from which we deduce that
	\begin{equation}\label{eq:utile}
		\sup_{0\leq t\leq T} \e\left[\left| \delta Y_t \right|^2 +\int_0^T \left(|\delta Y_s|^2 + |\delta Z_s|^2\right) \, ds \right] \leq 2\, \sup_{0\leq t\leq T}  \e\left[\int_t^T e^{as}\delta Y_s\, d\delta K_s\right].
	\end{equation}
	For any $s\in[0,T]$, denoting $v^n_s:=\left(u-y^n_s \right)_+$ where $y^n_s$ stands for $\e\left[Y^n_s\right]$, we have $dK^n_s = nv^n_s ds $ and 
	\begin{equation*}
		\e\left[\int_t^T e^{as} \delta Y_s\, d\delta K_s\right] = \int_t^T e^{as} \left[y^{n+1}_s-y^n_s\right] \left[(n+1) v^{n+1}_s  - nv^n_s\right] ds, \quad 0\le t \le T.
	\end{equation*}
	Moreover, we compute
	\begin{eqnarray*}
		\left[y^{n+1} -y^n\right] \left[(n+1) v^{n+1} - nv^n\right] &=&
		 \left[\left(u-y^n\right) - \left(u-y^{n+1})\right)\right] \left[(n+1) v^{n+1} - nv^n\right] \\
		&\le& -n |v^n|^2 +(2n+1) v^n v^{n+1} - (n+1) |v^{n+1}|^2 \;.
	\end{eqnarray*}
	But, we have
	\begin{equation*}
		-nx^2 + (2n+1) xy -(n+1) y^2 = -n \left(x-\left(1 + \dfrac{1}{2n}\right)y\right)^2 + \dfrac{y^2}{4n} \;, \qquad x,y\in\R\;,
	\end{equation*} 	
	so that combining the previous estimates with \eqref{eq:fonda}, we deduce
	\begin{equation*}
		\e\left[\int_0^T e^{as} \delta Y_s\, d\delta K_s\right] \leq \frac{1}{4n} \int_0^T |v^{n+1}_s|^2 ds \leq  \frac{C(\lambda,T)}{n^3} \;.
	\end{equation*}
	Plugging this estimate in \eqref{eq:utile}, it follows that
	\begin{equation*}
		\sup_{0\leq t\leq T} \e\left[ |\delta Y_t|^2 \right] + \e\left[ \int_0^T \left(|\delta Y_s|^2 + |\delta Z_s|^2\right) ds \right] \leq \frac{C(\lambda,T)}{n^3}.
	\end{equation*}
	Setting $\Delta_t K^n = K^n_T - K^n_t$ and reminding that $K^n$ is deterministic, observe that 
	\begin{equation*}
		 \Delta_t K^{n+1} - \Delta_t K^{n} =  \mathbb E[ \delta Y_t]  - \mathbb E\left[\int_t^T (f(s,Y^{n+1}_s,Z^{n+1}_s) -  f(s,Y^{n}_s,Z^{n}_s)) \, ds\right] \;,
		\end{equation*}
		 from which we deduce
		 \begin{equation*}
		 	\sup_{0\leq t\leq T}|\Delta_t K^{n+1} - \Delta_t K^{n} |\le \frac{C(\lambda,T)}{n^3}.
		 \end{equation*}
		 Since we have
		 \begin{equation*}
		 	\delta Y_t = \e \left(\int_t^T \left( f\left(s,Y^{n+1}_s,Z^{n+1}_s\right) -  f(s,Y^{n}_s,Z^{n}_s) \right) \, ds \: \Big|\: \m F_t\right) + \Delta_t K^{n+1} - \Delta_t K^n,
		 \end{equation*}
combining the above and Burkholder-Davis-Gundy inequality,  we conclude that  $(Y^n,Z^n,K^n)_n$ converges strongly to a limit $(Y, Z, K)$, namely
\begin{equation*}
	 \e\left[ \sup_{0\leq t\leq T} |Y^n_t- Y_t|^2 + \int_0^T | Z^n- Z_s|^2 ds \right] + \sup_{0\leq t\leq T} | K^n_t-K_t|^2 \longrightarrow_{n\rightarrow \infty} 0.
\end{equation*}

\paragraph{Step 3. Properties of the limit $(Y,Z,K)$}
 Passing to the limit the dynamics of $(Y^n,Z^n,K^n)_n$, remark that $(Y,Z,K)$ satisfies \eqref{eq:main_dyn}. Observe also that, by construction, $K$ is deterministic, nondecreasing with $K_0=0$. 
 Besides, the estimate \eqref{eq:fonda} directly implies that 
	\begin{equation*}
        \int_0^T |(u-\e[Y_t])_+|^2 dt   \,=\, \lim_{n\rightarrow\infty} \int_0^T |(u-\e[Y^n_t])_+|^2 dt \,=\, 0\;,
	\end{equation*}
	so that $\e[Y_t]\ge u$, for any $t\in[0,T]$. Finally from Lemma~\ref{en:analysis} below, since $(\e[Y^n],K^n)$ converges to $(\e[Y], K)$ in $\m C([0,T])$, we have 
	\begin{equation*}
		\lim_{n\to\infty} \int_0^T (\e[Y^n_t]-u)_+ dK^n_t = \int_0^T (\e[Y_t]-u)_+ dK_t,
	\end{equation*}
	and, on the other hand, 
	\begin{equation*}
        \int_0^T (\e[Y^n_t]-u)_+ dK^n_t   \,=n\, \int_0^T (\e[Y^n_t]-u)_+ (u-\e[Y^n_t])_+ dt \;=\; 0.
	\end{equation*}
 It follows that $(Y,Z,K)$ is the unique flat deterministic solution to the BSDE  \eqref{eq:main_dyn} with linear mean reflection \eqref{eq:main_constraint_linear}.
 \end{proof} 
 
 We now complete the argumentation by proving a rather elementary lemma, that we just used in the previous proof. 
 
 \begin{lemme}\label{en:analysis}
 	Let $(u^n)_{n\geq 1}$ and $(K^n)_{n\geq 1}$ be two convergent sequences of $\left(\m C_T, |\cdot |_\infty\right)$. We assume that, for each $n\geq 1$, $K^n$ is nondecreasing and we denote by $u$ and $K$ the corresponding limits of $(u^n)_n$ and $(K^n)_n$. We have
	\begin{equation*}
		\lim_{n\to\infty} \int_0^T u^{n}_t dK^n_t = \int_0^T u_tdK_t.
	\end{equation*}
 \end{lemme}
 
 \begin{proof}
	For any piecewise constant function $h$, we have
	\begin{align*}
		\int_0^T u^n_s dK^n_s - \int_0^T u_s dK_s & =  \int_0^T [u^n_s-u_s] dK^n_s + \int_0^T [u_s-h_s] dK^n_s + \int_0^T h_s dK^n_s \\
		& \quad - \int_0^T h_s dK_s + \int_0^T [h_s-u_s] dK_s,
	\end{align*}
	from which we deduce that
	\begin{align*}
			\left|\int_0^T u^n_s dK^n_s - \int_0^T u_s dK_s\right| &  \leq |u^n-u|_\infty |K^n|_\infty + |u-h|_{\infty} \left(|K^n|_\infty + |K|_\infty\right) \\
			&\quad + \left|\int_0^T h_s dK^n_s - \int_0^T h_s dK_s\right|.
	\end{align*}
	Since $h$ is piecewise constant, we have 
	\begin{equation*}
		\lim_{n\to\infty}\int_0^T h_s dK^n_s = \int_0^T h_s dK_s, \quad\text{and}\quad\limsup \left|\int_0^T u^n_s dK^n_s - \int_0^T u_s dK_s\right| \leq 2 \, |u-h|_\infty\, |K|_\infty,
	\end{equation*} 
	from which we get the result since piecewise constant functions on $[0,T]$ are dense in $\left(\m C_T, |\cdot |_\infty\right)$.
 \end{proof}

\section{ BSDE with general mean reflection} 
\label{sec:general_case}

We now turn to the general case where $x\longmapsto \ell(t,\omega,x)$ is non necessarily linear. We recall that we still work under Assumptions  ($H_\xi$)-($H_f$)-($H_\ell$) presented in Section \ref{sec:setup}.
In the same spirit as the approach presented in the previous section, we first construct explicitly a solution whenever the driver does not depend on $Y$ nor $Z$, and then tackle the general case via a Picard contraction argument. The construction of an explicit solution in the non linear case is less natural  and relies a lot on the use of the following operator: 
\begin{eqnarray*}
	L_t : \lp^2\left(\Fc_T\right) &\rightarrow& [0,\infty)  \\
	X &\mapsto& \inf\left\{ x\geq 0 : \e\left[\ell(t,x+X)\right] \geq 0 \right\}\;,
\end{eqnarray*}
defined for any $t\in[0,T]$. Since $\ell$ is of linear growth at infinity and $\e\left[\ell(t,\infty)\right]>0$, $L_t$ is well defined. Namely, $L_t(X)$ represents the minimal deterministic strength with which the random variable $X$ must be pushed upward in order to satisfy the constraint of interest at time $t$. In the previous linear case where $\ell:(t,x)\mapsto x-u_t$, we simply explicitly have $L_t : X \mapsto \left(\e\left[X\right]-u_t\right)_-$. \\

We first focus on the constant driver case and we then are able to tackle the general case. For this last framework, a Lipschitz property for the operator $L$ will be required. 

\subsection{ The constant driver case} 

In this section, we demonstrate the well posedness of the BSDE of interest in the constant driver case. As explained above, the operator $L$ plays an important role in order to build a solution to such BSDE. 

\begin{prop}\label{prop:existence_uniqueness_2}
 Let $C$ be a square integrable progressively measurable stochastic process or more generally in the space $\lp^2\left(\Omega;\lp^1(0,T)\right)$. 
 
 Then, the BSDE with mean reflection
\begin{equation}\label{eq:nov}
	Y_t  =\xi+\int_t^T C_s\, ds - \int_t^T Z_s\cdot dB_s + K_T-K_t,\qquad \e[\ell(t,Y_t)] \geq 0, \qquad 0\leq t\leq T,
\end{equation}
 has a unique square integrable deterministic flat solution. 
\end{prop}

 \begin{proof}
We derive the existence and uniqueness properties separately. 
 
\paragraph{Step 1. Existence}
 In order to solve \eqref{eq:nov}, let us define
\begin{equation*}
	\Psi_t := L_t\left(X_t\right),\text{ where }X_t = \e_t\left(\xi + \int_t^T C_s\, ds\right), \qquad 0\le t \le T\;.
\end{equation*}
Since $\ell$ is continuous in space, observe that 
\begin{equation}\label{temp123123}
	\e\left[\ell\left(t,X_t + \Psi_t\right)\right] \geq 0, \qquad 0\le t \le T\;.
\end{equation}
Let us now show that $\Psi$ is moreover continuous. Observe first that the map $x\longmapsto \e\left[\ell(t,x+X)\right]$ is continuous and strictly increasing. If $\e\left[\ell(t,X_t)\right]\leq 0$, since $\ell$ is continuous and has linear growth, for any  $x< L_t(X_t)<y $, one has
\begin{equation*}
	\lim_{s\to t} \e\left[\ell(s,x+X_{s})\right]=\e\left[\ell(t,x+X_t)\right] < 0 = \e\left[\ell(t,L_t(X_t)+X_t)\right] < \e\left[\ell(t,y+X_t)\right]=\lim_{s\to t} \e\left[\ell(s,y+X_{s})\right].
\end{equation*}
Then, if $|s-t|$ is small enough, $\e\left[\ell(s,x+X_{s})\right]<0$, $\e\left[\ell(s,y+X_{s})\right]>0$ and $x\leq L_s(X_s) \leq y$.

If  $\e\left[\ell(t,X_t)\right] >0 $, $L_t(X_t)=0$, and $\lim_{s\to t} \e\left[\ell(s,X_{s})\right]=\e\left[\ell(t,X_t)\right]>0$. If $|s-t|$ is small enough, $\e\left[\ell(s,X_{s})\right]>0$ and $L_s(X_s)=0$.

\smallskip

We are now in position to define the continuous process $K$ by 
\begin{equation*}
	K_t := \sup_{0\leq s\leq T} \Psi_s - \sup_{t\leq s\leq T} \Psi_s\;,  \; \qquad \mbox{so that } \qquad K_T-K_t = \sup_{t\leq s \leq T} \Psi_s \;, \qquad 0\le t \le T\;.
\end{equation*}
Observe that $K$ is deterministic, non decreasing with $K_0=0$. Given this process $K$, let $(Y,Z)$ be the unique solution to the classical BSDE with the dynamics of \eqref{eq:nov}. Then, since $x\longmapsto \ell(t,x)$ is non decreasing, we deduce from \eqref{temp123123} that 
\begin{equation}
	\label{eq:pos}
	\e\left[\ell(t,Y_t)\right] = \e\left[\ell\left(t,X_t + K_T-K_t\right)\right] = \e\left[\ell\left(t,X_t + \sup_{t\leq s\leq T} \Psi_s\right)\right] \geq \e\left[\ell\left(t,X_t + \Psi_t\right)\right] \geq 0.
\end{equation}
Hence,  $(Y,Z,K)$ is a deterministic solution to the BSDE with weak reflexion \eqref{eq:nov}. 

Let now verify that it is also flat. 
%
%
By definition of $K$, observe that $\sup_{t\leq s\leq T} \Psi_s = \Psi_t$ $dK_t$-a.e. and $\ind_{\Psi_t = 0} = 0$ $dK_t$-a.e. Thus, by \eqref{eq:pos}, we compute
\begin{equation*}
	\int_0^T \e\left[\ell(t,Y_t)\right] \, dK_t = \int_0^T \e\left[\ell(t,X_t + \Psi_t)\right] \, dK_t = \int_0^T \e\left[\ell(t,X_t + \Psi_t)\right]\ind_{\Psi_t >0} \, dK_t.
\end{equation*}
 Besides, since $\ell$ is continuous in space, we have $\e\left[\ell(t,X_t + \Psi_t)\right]=0$ as soon as $\Psi_t>0$, so that 
\begin{equation*}
	\int_0^T \e\left[\ell(t,X_t + \Psi_t)\right]\ind_{\Psi_t >0} \, dK_t = 0,
\end{equation*}
and $(Y,Z,K)$ is a flat solution.\\
%
%
%
%
%
%
%
%
%
%
%

\paragraph{Step 2. Uniqueness}
Let $(Y^1,Z^1,K^1)$ and $(Y^2,Z^2,K^2)$ be two deterministic flat solutions to the BSDE with mean reflexion \eqref{eq:nov}.
 We work towards a contradiction and suppose that there exists $t_1<T$ such that 
   $$K^1_T-K^1_{t_1}> K^2_T - K_{t_1}^2.$$
%
Setting $t_2$ as the first time $t$ after $t_1$ such that $K^1_T-K^1_{t}=K^2_T - K_{t}^2$, we observe that 
$$K^1_T-K^1_{t}>K^2_T - K_{t}^2,\quad t_1\le t < t_2 \;.$$
Since $\ell$ is strictly increasing, this implies that 
$$
\e[\ell(t,X_t+K_T^1-K_t^1)]>\e[\ell(t,X_t+K_T^2-K_t^2)]\ge 0,\quad t_1\le t < t_2 \;.
$$
But $(Y^1,Z^1,K^1)$ is a flat solution and hereby  
$$
\int_{t_1}^{t_2} \e[\ell(t,X_t+K_T^1-K_t^1)]dK_t^1 =0\;,
$$
so that we must have $dK^1=0$ on the interval $[t_1,t_2]$. We deduce that 
$$K_T^1-K^1_{t_2}=K_T^1-K_{t_1}^1>K_T^2-K_{t_1}^2\ge K^2_T-K^2_{t_2}\;,$$
which contradicts the definition of $t_2$. Hence $K^1=K^2$ and the uniqueness of solution to classical BSDEs directly implies that $(Y^1,Z^1,K^1)$ coincides with  $(Y^2,Z^2,K^2)$.
\end{proof}

\subsection{ Existence and uniqueness for the general case} 

Now that the well posedness for constant driver is established, we can focus on the BSDE  \eqref{eq:main_dyn} with mean reflexion \eqref{eq:main_constraint} in full generality. In order for the solution to be well defined, we will require a Lipschitz property of the operator $L$, that we present in the following additional Assumption: 
 \begin{itemize}
 \item[($H_L$)] The operator $L_t$ is Lipschitz continuous for the $\lp^1$-norm, uniformly in time: namely there exists a constant $C\geq 0$ such that
 \begin{equation*}
|L_t(X)-L_t(Y)|\le C\;\mathbb \e\left[|X-Y|\right] , \qquad 0\le t \le T\;, \;\; X,Y \in \lp^2\left(\Fc_t\right)\;.
\end{equation*}
\end{itemize} 

We are now in position to state the main result of the paper, providing the well-posedness of BSDEs with mean reflexion. 

\begin{thm}\label{thm: main}
In addition to the running assumptions ($H_\xi$)-($H_f$)-($H_\ell$), let us moreover assume that ($H_L$) is satisfied.
Then, there exists a unique deterministic flat solution $(Y,Z,K)\in\Sc^2\times \Hc^2\times\Ac^2_D$ to the BSDE \eqref{eq:main_dyn} with mean reflexion  \eqref{eq:main_constraint}.
%
\end{thm}

\begin{proof}
	Let us consider $\sigma$ and $\tau$ in the time interval $[0,T]$ with $\sigma\leq \tau$. Given $Y_\tau \in \lp^2\left(\m F_\tau\right)$, $\{U_t\}_{\sigma\leq t\leq \tau}\in\m S^2$ and $\{V_t\}_{\sigma\leq t\leq \tau}\in\m H^2$, Proposition~\ref{prop:existence_uniqueness_2} ensures the existence of a triple of processes $\{(Y_t,Z_t,R_t)\}_{\sigma\leq t\leq \tau}$ solution to the BSDE with mean reflexion
	\begin{gather*}
		Y_t= Y_\tau + \int_t^\tau f(s,U_s,V_s)\, ds - \int_t^\tau Z_s\cdot dB_s + R_t, \qquad \sigma\leq t\leq \tau, \\
		\e\left[\ell(t,Y_t)\right] \geq 0, \quad \sigma\leq t\leq \tau, \qquad \int_\sigma^\tau \e\left[\ell(t,Y_t)\right]\, dR_t = 0\,,
	\end{gather*}
	where we conveniently denoted $R_. = K_\tau-K_.$. In this setting, $R$ is non increasing with $R_\tau=0$ and, for $\sigma\leq t\leq \tau$,
	\begin{equation}
		\label{eq:repR}
		R_t = \sup_{t\leq s \leq \tau}\,L_s(X_s), \quad\text{with}\quad X_t = \e\left[Y_\tau + \int_t^\tau f(s,U_s,V_s)\, ds \:\Big|\: \m F_t\right].
	\end{equation}
	Let $(Y',Z',R')$ be the solution associated to $(U',V')$ and the same $Y_\tau$.
	
	We have, with usual notations, 
	\begin{equation*}
		\delta Y_t = \e\left[\int_t^\tau \left[f(s,U_s,V_s)-f(s,U'_s,V'_s)\right]\, ds \:\Big|\: \m F_t\right] + \delta R_t, \quad \sigma\leq t\leq \tau,
	\end{equation*}
	from which we deduce immediately, since $f$ is assumed to be Lipschitz, that
	\begin{equation*}
		\e\left[\sup_{\sigma\leq t\leq \tau}|\delta Y_t|^2 \right] \leq C(\lambda) \, \e\left[ \left(\int_\sigma^\tau \left(|\delta U_s|+|\delta V_s|\right)ds\right)^2\right]+ \sup_{\sigma \leq t \leq \tau} \left|\delta R_t\right| ^2.
	\end{equation*}
	Besides, since ($H_L$) holds, we deduce from the representation \eqref{eq:repR} together with $\delta Y_\tau=0$ and the Lipschitz property of $f$ that, for $\sigma\leq t\leq\tau$,
	\begin{align*}
		\left|\delta R_t\right| & \leq \left|\sup_{t\leq s\leq\tau} L_s(X_s) - \sup_{t\leq s\leq\tau} L_s(X'_s)\right| \leq \sup_{t\leq s\leq\tau} \left| L_s(X_s)-L_s(X'_s) \right| \leq \sup_{t\leq s\leq\tau} \e\left[\left|\delta X_s\right|\right] \\
		& \leq C(\lambda)\, \e\left[ \int_\sigma^\tau \left(|\delta U_s|+|\delta V_s|\right)ds\right] \;.
	\end{align*}
	Combining the previous estimates together with the Cauchy Schwartz inequality, we deduce
	\begin{equation*}
		\e\left[\sup_{\sigma\leq t\leq \tau}|\delta Y_t|^2 \right] \leq C(\lambda) \, \e\left[ \left(\int_\sigma^\tau \left(|\delta U_s|+|\delta V_s|\right)ds\right)^2\right],
	\end{equation*}
	and writing
	\begin{equation*}
		\int_\sigma^\tau \delta Z_s\cdot dB_s = \delta Y_\tau-\delta Y_\sigma + \delta R_\tau-\delta R_\sigma+\int_\sigma^\tau \left[f(s,U_s,V_s)-f(s,U'_s,V'_s)\right]\, ds
	\end{equation*}
	we finally have
	\begin{align}
		\nonumber
		\e\left[\sup_{\sigma\leq t\leq \tau}|\delta Y_t|^2 + \int_\sigma^\tau |\delta Z_s|^2 \, ds\right] & \leq C(\lambda) \, \e\left[ \left(\int_\sigma^\tau \left(|\delta U_s|+|\delta V_s|\right)ds\right)^2\right] , \\
		& \leq C(\lambda)\, \left(\tau-\sigma\right)\max\left(1, \tau-\sigma\right) \, \e\left[\sup_{\sigma\leq t\leq \tau}|\delta U_t|^2 + \int_\sigma^\tau |\delta V_s|^2 \, ds\right].
		\label{eq:smalltime}
	\end{align}
	Of course, this inequality shows that the BSDE~\eqref{eq:main_dyn} with mean reflexion~\eqref{eq:main_constraint} has a unique solution whenever $T$ is small enough.
	
	To cover the general case, let us pick $n\geq 1$ such that $C(\lambda) \min(T,T^2) / n^2 < 1$. For $i=0,\ldots, n$, let us set $T_i:=iT/n$. Starting from the interval $[T_{n-1},T_n]$ and $Y_{T_n}=\xi$, let, for $i=n,\ldots,1$, $(Y^i,Z^i,R^i)$ the unique solution to the BSDE with mean reflexion
	\begin{gather*}
		Y^i_t= Y^{i+1}_{T_i}+ \int_t^{T_i} f\left(s,Y^i_s,Z^i_s\right)\, ds - \int_t^{T_i} Z^i_s\cdot dB_s + R^i_t, \qquad \e\left[\ell\left(t,Y^i_t\right)\right] \geq 0 \qquad T_{i-1}\leq t\leq T_i, \\
	 \int_{T_{i-1}}^{T_i} \e\left[\ell\left(t,Y^i_t\right)\right]\, dR^i_t = 0, \qquad R^i \text{ continuous and non increasing on } [T_{i-1}, T_{i}] \text{ with }R^i_{T_i}=0.
	\end{gather*}
	Let us define $(Y,Z,R)$ on $[0,T]$ by setting 
	\begin{equation*}
		Y_t = Y^1_0 \ind_{{0}}(t) + \sum_{i=1}^n Y^i_t \ind_{]T_{i-1},T_i]}(t), \qquad Z_t = \sum_{i=1}^n Z^i_t \ind_{]T_{i-1},T_i[}(t),
	\end{equation*}
	and $R_t=R^n_t$ on $[T_{n-1},T_{n}]$ and, for $i=n-1,\ldots 1$, $R_t=R^{i}_t + R_{T_i}$ on $[T_{i-1},T_i]$. Since $R^i_{T_i}=0$, $R$ is continuous and non increasing. Finally, let us define $K_t = R_0-R_t$ to get a non decreasing continuous function with $K_0=0$. Since $R_T=0$, $K_T=R_0$ and $R_t=K_T-K_t$.
	
	It is plain to check that $(Y,Z,K)$ is a solution to the BSDE~\eqref{eq:main_dyn} with mean reflexion~\eqref{eq:main_constraint}. Uniqueness follows from the uniqueness on each small interval.
\end{proof}

It is worth noticing that the previous assumption $(H_L)$ is automatically satisfied as soon as $\ell$ is a bi-Lipschitz function in $x$. More precisely, we consider the following alternative assumption on $\ell$:
\begin{itemize}
	\item [($H_{b\ell}$)] The loss function $\ell : \Omega\times[0,T]\times\rset \fl \rset$ is a measurable map with respect to $\m F_T\times\m B([0,T])\times \m B(\rset)$ and there exists $0< c_l\leq C_l$ such that, $\p$-a.s.,
\begin{enumerate}
	\item $\forall y\in\rset$, $t\longmapsto \ell(t,y)$ is continuous,
	\item $\forall t\in[0,T]$, $y\longmapsto \ell(t,y)$ is strictly increasing,
	\item $\forall t\in[0,T]$, $\forall y\in\rset$, $|\ell(t,y)| \leq C_l(1+|y|)$.
	\item $\forall t\in[0,T]$, 
	\begin{equation}\label{bilip}
		{c_{\ell}} |x-y| \leq |\ell(t,x)-\ell(t,y) | \leq C_{\ell} |x-y|\;, \qquad  x,y\in\rset\;,
	\end{equation}
\end{enumerate} 
\end{itemize}

\begin{lemme}
Assume $(H_{b\ell})$. Then both Assumptions ($H_\ell$) and ($H_L$) hold.
\end{lemme}

\begin{proof} 
	Observe first that $(H_{b\ell})$ implies directly that $(H_\ell)$ holds. Fix now $t\in[0,T]$ and let $X$ and $Y$ be two random variables in $\lp^2\left(\Fc_T\right)$.
 
 Since $\ell$ is non decreasing, the lower bound of \eqref{bilip} gives
\begin{eqnarray*}
\ell\left(t,L_t(X)+\frac{C_{\ell}}{c_{\ell}}\e\left[|X-Y|\right]+Y\right) &\ge& c_{\ell} \frac{C_{\ell}}{c_{\ell}}\e\left[|X-Y|\right] + \ell(t,L_t(X)+Y),
\end{eqnarray*}
and using the upper bound we get
\begin{equation*}
	\ell(t,L_t(X)+Y) \geq \ell(t,L_t(X)+X) -C_l |X-Y|, 
\end{equation*}
from which it follows
\begin{equation*}
	\ell\left(t,L_t(X)+\frac{C_{\ell}}{c_{\ell}}\e\left[|X-Y|\right]+Y\right) \geq \ell(t,L_t(X)+X) - C_l\, |X-Y| +C_l\, \e\left[|X-Y|\right]
\end{equation*}
Since $\e\left[\ell(t,X+L_t(X))\right]\geq 0$, we obtain by taking the expectation of the previous inequality
\begin{equation*}
	\e\left[\ell\left(t,L_t(X)+\frac{C_{\ell}}{c_{\ell}}\e\left[|X-Y|\right]+Y\right)\right] \geq 0.
\end{equation*}
By definition of $L_t(Y)$, this directly implies that 
$$ 
L_t(Y)\le L_t(X)+\frac{C_{\ell}}{c_{\ell}}\e\left[|X-Y|\right] \;.
$$
By symmetry of $X$ and $Y$, we conclude that
$$|L_t(X)-L_t(Y)|\le \frac{C_{\ell}}{c_{\ell}}\,\e\left[|X-Y|\right].$$
\end{proof}

As a byproduct, we have the following result.

\begin{cor}\label{en:TheMain}
	Let ($H_\xi$), ($H_f$) and ($H_{b\ell}$) hold. 

	Then, there exists a unique deterministic flat solution $(Y,Z,K)\in\Sc^2\times \Hc^2\times\Ac^2_D$ to the BSDE \eqref{eq:main_dyn} with mean reflexion  \eqref{eq:main_constraint}.
\end{cor}

%
%
%
%
%
%
%
%
%
%

%
%


\section{Minimality of the deterministic flat solution}
\label{sec:minimal}

 Let us recall that for classical reflected BSDE, the Skorokhod condition ensures the minimality of the enhanced solution in the class of all supersolutions to the reflected BSDE. By minimality, we refer to minimality in terms of the $Y$-component of the solution. The Skorokhod condition indicates that the compensator $K$ only pushes the solution when the condition is binding, i.e. only when it is really necessary. In this spirit, we chose in this paper to look towards solutions to BSDEs with mean reflection which satisfy the corresponding flatness condition \eqref{main_flat}.\\
 
 Now, that the existence of a unique deterministic flat solution to the BSDE \eqref{eq:main_dyn} with mean reflexion  \eqref{eq:main_constraint} has been established, it is natural to wonder if this flatness condition \eqref{main_flat} also implies the minimality  among all the deterministic  solutions. Since the constraint is given in expectation instead of pointwisely, it is not obvious that only the condition at time $t$ determines the minimal upward kick to apply on the solution at time $t$. Under additional assumption on the structure of the driver function $f$, we are able to verify that such minimality property is indeed satisfied.


\begin{thm}\label{minimal}
Suppose that the driver function $f$ is of the form
\begin{equation}\label{struct_f}
f : (t,y,z) \mapsto a_t y+h(t,z) \;, 
\end{equation}
where $a$ is a deterministic and bounded measurable function. If $\ell$ is strictly increasing, a deterministic flat solution $(Y,Z,K)$ is minimal among all the deterministic solutions.
\end{thm}

\begin{proof}
Let  $(Y,Z,K)$ be a deterministic flat solution, and $(Y',Z',K')$ be any deterministic solution. We want to prove that $Y\le Y'$. We first focus on the particular case where the driver does not depend on $y$ and then tackle the general case where $f$ is given by \eqref{struct_f}. \\

\noindent {\bf Step 1. Driver of the form $f(t,z)$.} \\
 Since the driver function $f$ does not depend on $y$, the processes $(Y-(K_T-K),Z)$ and $(Y'-(K'_T-K'),Z')$ are both solutions of the same classical BSDE, and we deduce that 
\begin{equation}\label{tmptmp}
Y_t-(K_T-K_t)=Y'_t-(K'_T-K'_t), \qquad 0\le t\le T \;.
\end{equation}
Hereby, proving that $Y\le Y'$ boils down to showing that $K_T-K \le K'_T-K'$. We work towards a contradiction and suppose the existence of $t_1<T$ such that $$K_T-K_{t_1}>K'_T-K'_{t_1}.$$
Let $t_2$ be the first time such that $K_T-K_. \ge K_T'-K'_.$. Obviously $t_2$ is a deterministic time smaller than $T$ and by continuity of $K$ and $K'$, we get $K_T-K_{t_2} = K'_T-K'_{t_2}$ and 
 \begin{equation}
 K_T-K_{t} > K'_T-K'_{t}\;, \qquad t_1\le t < t_2\;.
 \end{equation}
 We deduce from \eqref{tmptmp} that $Y> Y'$, on  $[t_1,t_2)$, and the strict monotony of $\ell$ implies
%
$$\mathbb E[\ell(t,Y_t)]>\mathbb E[\ell(t,Y'_t)]\ge 0\;, \qquad t_1\le t \le t_2 \,.$$
Since $Y$ is a flat solution, we have $\int_0^T \mathbb E[\ell(Y_s)]dK_s=0$ and we deduce that $dK_t=0$, for $t\in [t_1,t_2)$.
Therefore,
$$K_T'-K'_{t_1}<K_T-K_{t_1}=K_T-K_{t_2}=K'_T-K'_{t_2}$$
which is a contradiction since $K'$ must be non decreasing.

\noindent {\bf Step 2. Driver of the form \eqref{struct_f}.} \\
Let us denote $A_t:= \int_0^t a_sds$ for $0\le t \le T$. Making the following transformation 
$$\tilde{Y}_t=e^{A_t}Y_t,\quad \tilde{Z}_t=e^{A_t}Z_t, \quad \tilde{K}_t=e^{A_t}K_t,$$
we verify easily that $(\tilde{Y},\tilde{Z},\tilde{K})$ is a flat deterministic solution to the BSDE with mean reflection associated to the parameters
$$\tilde \xi =  e^{A_T} \xi \;, \qquad \tilde f(t,z)=e^{A_t} f(t,e^{-{A_t}}z) \; \quad \mbox{and}\quad  \tilde \ell(t,y)=\ell(t,e^{-{A_t}}y) \;.$$
According to the previous step $\tilde Y$ is minimal within the class of deterministic solutions, and $Y$  inherits this property by a straightforward argument.
\end{proof}

\begin{rem}
As a by-product, this proof provides an alternative argument in order to derive the uniqueness of the flat deterministic solution of BSDEs with mean reflexion and driver of the form  \eqref{struct_f}. It is in fact a generalization of the proof presented in Proposition \ref{prop:existence_uniqueness_2}  for the constant driver case. \\
\end{rem}


We now exhibit an example to show that if we allow $K$ to be random, then there exists no minimal flat solution to the BSDE with mean reflection. This argument strengthens our choice to focus solely in this paper on so-called deterministic solutions.

  For this purpose, let consider BSDE with mean reflection 
	\begin{gather*}
		Y_t = \xi - \int_t^T \gamma \, ds - \int_t^T Z_s\cdot dB_s + K_T-K_t, \quad 0\leq t\leq T ,\\
		\e[Y_t] \geq u,\quad 0\leq t \leq T, \qquad \int_0^T \left(\e[Y_t]-u\right) dK_t = 0,
	\end{gather*}	
	 with $\gamma>0$,  and the terminal condition $\xi$ such that $u < \e[\xi] < u+\gamma T$.
	 
	 As detailed in Section \ref{sec:llin}, the deterministic flat solution to the BSDE is given by 
	\begin{gather*}
		Y_t = \e\left(\xi\:|\: \m F_t\right) - \gamma (T-t) + {\left(\e[\xi]-\gamma (T-t)-u\right)^-}, 
	\end{gather*}
and $K_t = \gamma (t\wedge t^*)$, where we pick $t^*$ to verify 
$$	\qquad \e[\xi] - \gamma (T-t^*) = u.$$
	 Starting from the previous solution, for $\alpha \in\rset$, we set
	\begin{equation*}
		M^\alpha_t := \exp\left(\alpha B_t - \alpha^2 t/2\right)\quad \mbox{ and }\quad \quad K^\alpha_t := \int_0^t M^\alpha_s\, dK_s \;, \qquad 0\le t \le T\,.
	\end{equation*}
	 Given $K^\alpha$, let $\left(Y^\alpha,Z^\alpha\right)$ be the solution to the classical BSDE
	\begin{equation*}
		Y^\alpha_t = \xi - \int_t^T \gamma \, ds - \int_t^T Z^\alpha_s dB_s + K^\alpha_T - K^\alpha_t, \quad 0\leq t\leq T.
	\end{equation*}
 Then $\left(Y^\alpha,Z^\alpha,K^\alpha\right)$ is still a flat solution to the reflected BSDE,  see Remark \ref{rem_uniqueness} in Section \ref{sec:llin}.\\
		 Let us suppose the existence of a minimal solution $(\bar Y,\bar Z,\bar K)$ and look towards a contradiction.  We have
		\begin{eqnarray*}
			\bar Y_t & \leq & Y^\alpha_t = \e_t(\xi) - \gamma (T-t) +\e_t\left( \int_t^T M^\alpha_s dK_s\right)
			\; = \; \e_t(\xi) - \gamma (T-t) + M^\alpha_t (K_T-K_t)\,,
		\end{eqnarray*}
		for $t>0$. As a byproduct, sending $\alpha$ to $+\infty$, we deduce $\bar Y_t \leq \e_t(\xi)-\gamma(T-t)$ for $t>0$, and in particular
		\begin{equation*}
			\forall t>0, \qquad \e\left[\bar Y_t\right] \leq \e\left[\xi\right] -\gamma (T-t).
		\end{equation*}
		Since $\e\left[\xi\right]-\gamma T < u$, for $t>0$ small enough, $\e\left[\bar Y_t\right]<u$. The constraint is not satisfied and we get a contradiction.

\section{Extension and application}
\label{sec:application}

 Interpreting $Y$ as the value of a portfolio, the constraint \eqref{eq:main_constraint} imposes at any date $t$ a constraint on the distribution of $Y_t$, seen from time $0$. The form of constraint that we considered so far is the expectation of a loss function. From a financial point of view, an investor may be required  to control the risk of any admissible portfolio. In order to measure the underlying risk of a portfolio, the natural tool in the mathematical finance literature are the so-called  risk measures, see e.g. \cite{ADEH}. We emphasize in this section how our framework of study allows to encompass such type of running static risk measure constraint. Then, we present an application for the problem of super hedging a claim under a given running risk measure constraint. \\

\subsection{ BSDE with risk measure reflection} \label{subsec_risk}

 For a fixed $t$, a static risk measure is a map   $\rho(t,.) : L^2(\m F_t) \fl \rset$ satisfying $\rho(t,0)=0$ together with
	\begin{itemize}
		\item Monotonicity:  $X\leq Y \Longrightarrow \rho(t,X) \geq \rho(t,Y)$, for $X,Y\in L^2(\m F_t)$ \,;
		\item Translation invariance: $\rho(t,X+m) = \rho(t,X) - m$,   for $X\in L^2(\m F_t)$ and $m\in\R$ \,.
	\end{itemize}
Hereby, for a given $t\in[0,T]$, $\rho(t,X)$ is a real number which measures the risk associated to the wealth random variable $X$. Risk measures can similarly be characterized by their so-called acceptance set, which defines as 
 \begin{equation*}
	\mathcal{A}_\rho^t = \{ X\in L^2(\m F_t) : \rho(t,X) \leq 0 \}.
 \end{equation*}
 Similarly,  given a set $\mathcal{A}^t$, one can define a static risk measure by setting
\begin{equation*}
	\rho(t,X) = \inf\{ m\in\rset : m+X \in \mathcal{A}^t \},
\end{equation*}
so that the acceptance set  $\mathcal{A}^t$  and the risk measure $\rho(t,.)$ share a one to one correspondence.  For a given collection of static risk measures $(\rho(t,.))_t$, a wealth process $Y$ will be considered admissible in our framework as soon as it satisfies 
 \begin{eqnarray}\label{constraint_riskmeasure} 
 \rho(t,Y_t) \le q_t\;, \qquad 0\le t \le T\;,
 \end{eqnarray}
 where $q$ is a given time indexed deterministic benchmark. For example, the risk measuring tool of $\rho$ could simply not depend on time, but be compared to the deterministic benchmark $q$, which evolves with time, by either tightening or relaxing  the constraint.  We now look towards solutions of BSDEs subject to the additional constraint \eqref{constraint_riskmeasure}. In the same spirit as above, a flat solution to such type of BSDE will be required to satisfy 
 \begin{eqnarray}\label{flatness_riskmeasure} 
 \int_0^T [q_t-\rho(t,Y_t)] dK_t=0 \;.
 \end{eqnarray}

 The next theorem indicates that we are able to consider BSDEs under risk measure constraint of the form \eqref{constraint_riskmeasure}, in a similar fashion as the one developed in the previous sections.

		\begin{thm}
			Let $\rho(t,.) : [0,T]\times\lp^2\fl \rset$ be a collection of monotonic and translation invariant risk measures, which are continuous with time and Lipschitz in space, i.e.
			\begin{eqnarray*}
			|\rho(t,X)-\rho(t,Y)|\le C \e [|X-Y|] \;, \qquad 0\le t \le T\;, \;\; X,Y \in \lp^2\left(\Fc_t\right)\;.
			\end{eqnarray*}
			If we are moreover given a  continuous deterministic benchmark $q$ and $\xi$ satisfies $\rho(T,\xi)\le q_T$, 
			 then the "\textit{BSDE with risk measure reflection}" 
		\begin{gather*}
			Y_t = \xi + \int_t^T f(s,Y_s,Z_s)\, ds - \int_t^T Z_s\cdot dB_s + K_T-K_t, \quad 0\leq t\leq T \\
			\rho(t,Y_t)\leq q_t,\quad 0\leq t \leq T, \qquad \int_0^T [q_t-\rho(t,Y_t)] dK_t = 0.
		\end{gather*}			
			admits a unique deterministic flat solution.\\ Besides, if $f$ satisfies \eqref{struct_f}, the deterministic flat solution is minimal among all deterministic solutions.
		\end{thm}

 \begin{proof}
 The reasoning  simply follows the arguments of Proposition \ref {prop:existence_uniqueness_2},  Theorem \ref{thm: main} and Theorem \ref{minimal}. The main distinction is that the map $L_t$ is replaced by the risk measure $\rho(t,.)-q_t$, for any $t\in[0,T]$. Besides, the translation invariance property conveniently replaces the strict monotonicity of $\ell$ in the proofs. 
 \end{proof}


Typical examples considered in the literature are coherent risk measures of the form 
			\begin{equation*}
				\rho(t,X) = \sup\{ \e^{\mathbb{Q}}\left[-X\right] : \mathbb{Q} \in \m Q_t\} \;, 
			\end{equation*}
where $\m Q_t$ is a set of probabilities absolutely continuous w.r.t. $\p$. As soon as the set of probability change densities is bounded, $\rho(t,)$ is Lipschitz. This is particular the case for the classical Expected Shortfall risk measure, defined as 
			\begin{eqnarray*}
				\rho^{ES}_\alpha(t,X) := \frac{1}{\alpha_t} \int_0^{\alpha_t} VaR_{s}(X) ds\;,   
			\end{eqnarray*}
			where $\alpha_t\in(0,1)$ denotes a given precision level and $VaR_s$ is the Value at Risk of level $s$. Indeed, the Expected Shortfall (or AVaR) rewrites also this way
%
			\begin{equation*}
				\rho^{ES}_\alpha(t,X) = \sup\left\{ \e^{\mathbb{Q}}\left[-X\right] :  \frac{d\mathbb{Q}}{d\p} \leq \frac{1}{\alpha_t} \right\} \;.
			\end{equation*}

\subsection{Application to super hedging under risk constraint} \label{subsec_application}   

 We now turn to an application in mathematical finance and consider a stock market endowed with a Bond with deterministic interest rate $r$ and a vector of $d$ stocks with dynamics
		\begin{equation*}
			d S_t \,=\, S_t \left( \mu_t dt + \sigma_t dB_t  \right)\;, \qquad 0\le t \le T\;,
		\end{equation*}
  where the drift $\mu$ and the volatility $\sigma$ are square integrable predictable processes. We assume that $\sigma_t\sigma_t'-\varepsilon I\succeq 0$
for some $\varepsilon>0$, in order to ensure the completeness of the market. 
For a given initial capital $x$, we consider portfolios $X^{x,\pi,K}$ driven by a consumption-investment strategy $(\pi,K)$, and whose dynamics are given by 
\begin{eqnarray*}
dX^{x,\pi,K}_t &=& X^{x,\pi,K}_t \left(r_t dt+ (\mu_t-r_t{\bf 1}) '\pi_t \frac{d S_t}{S_t} \right)  - dK_t \;, \\
 &=& r_t X^{x,\pi,K}_t dt + (\mu_t-r_t{\bf 1})'\pi_tdt+\pi_t'\sigma_t dB_t-dK_t\,,\qquad 0\le t \le T \,.
\end{eqnarray*}
 Using such portfolios, a financial engineer is willing to hedge a possibly non Markovian claim $\xi\in L^2(\Fc_T)$. For regulatory purposes, the  risk management department of his financial institution imposes him restrictions on the class of admissible investment strategies. Namely, a portfolio wealth process $X^{x,\pi,K}$ is considered admissible if and only if it satisfies  the following constraint : 
			\begin{equation*}
				\rho^{ES}_\alpha(t,X^{x,\pi,K}_t) \le  q_t \;,\qquad  0\le t \le T\;,
			\end{equation*}
 where $(\alpha,q)$ are a time indexed collection of deterministic quantile and level benchmarks. These benchmarks can for example be chosen in such a way that the constraint becomes either tighter or weaker, as we approach the maturity $T$. In such a case, the careful investor is looking for the super hedging price 
 \begin{equation*}
 Y_0 = \inf\{ x\in\R \;, \quad  \exists (\pi,K)\in\Ac\,, \quad s.t. \;\;  X^{x,\pi,K}_T\ge \xi  \;\;\;\; \mbox{ and } \;\;\;\; \rho^{ES}_\alpha(t,X_t) \le  q_t \;,\quad \forall \;\; t\in[0,T]   \} \;,
 \end{equation*}
 and associated consumption-investment strategy. Applying the results of this paper, we deduce that, if the investor restricts to deterministic consumption strategies, $Y_0$ is well defined as the starting point of the unique deterministic flat solution to the following BSDE with risk measure reflection 
 \begin{eqnarray*}
			Y_t = \xi + \int_t^T  r_t X^{x,\pi,K}_t dt + (\mu_t-r_t{\bf 1})' \sigma_t^{-1} Z_t^\top \, ds - \int_t^T Z_s\cdot dB_s + K_T-K_t, \quad 0\leq t\leq T, \\
			\rho^{ES}_\alpha(t,Y_t)\leq q_t,\quad 0\leq t \leq T, \qquad \int_0^T [q_t-\rho^{ES}_\alpha(t,Y_t)] dK_t = 0.
\end{eqnarray*}
 Indeed, the driver function satisfies \eqref{struct_f}, so that the flat solution is minimal among all deterministic ones. 

\end{document}